\documentclass[10pt]{amsart}
\usepackage{amsmath,latexsym}

\usepackage{color}
\usepackage[all,curve,frame]{xy}
\definecolor{Myblue}{rgb}{0.0,0,0.9}
\definecolor{Mygreen}{rgb}{0.2,1,0}

\newtheorem{thm}{Theorem}[section]

\newtheorem{prop}[thm]{Proposition}
\newtheorem{cor}[thm]{Corollary}

\theoremstyle{definition}

\theoremstyle{definition}

\numberwithin{equation}{section}

\newcommand{\benu}{\begin{enumerate}}
\newcommand{\enu}{\end{enumerate}}

\newcommand{\bema}{\left[\begin{array}}
\newcommand{\ema}{\end{array}\right]}

\newcommand{\End}{\operatorname{End}}

\newcommand{\Ext}{\operatorname{Ext}}

\newcommand{\HVCenter}[1]{\setbox 0=\hbox{#1}%
        \dimen0=\wd0%
        \dimen1=\ht0%
        \divide\dimen0 by 2%
        \divide\dimen1 by 2%
        \hskip -\dimen0%
        \lower \dimen1%
        \box0%
        \hskip -\dimen0}
\newcommand{\HBCenter}[1]{\setbox 0=\hbox{#1}%
        \dimen0=\wd0%
        \dimen1=\ht0%
        \divide\dimen0 by 2%
        \hskip -\dimen0%
        \box0%
        \hskip -\dimen0}
\newcommand{\LTCenter}[1]{\setbox 0=\hbox{#1}%
        \dimen1=\ht0%
        \lower \dimen1%
        \box0%
        \hskip -\dimen0}
\newcommand{\HTCenter}[1]{\setbox 0=\hbox{#1}%
        \dimen0=\wd0%
        \dimen1=\ht0%
        \divide\dimen0 by 2%
        \hskip -\dimen0%
        \lower \dimen1%
        \box0%
        \hskip -\dimen0}
\newcommand{\RTCenter}[1]{\setbox 0=\hbox{#1}%
        \dimen0=\wd0%
        \dimen1=\ht0%
        \hskip -\dimen0%
        \lower \dimen1%
        \box0%
        \hskip -\dimen0}
\newcommand{\RBCenter}[1]{\setbox 0=\hbox{#1}%
        \dimen0=\wd0%
        \dimen1=\ht0%
        \hskip -\dimen0%
        \box0%
        \hskip -\dimen0}
\newcommand{\RVCenter}[1]{\setbox 0=\hbox{#1}%
        \dimen0=\wd0%
        \dimen1=\ht0%
        \divide\dimen1 by 2%
        \hskip -\dimen0%
        \lower \dimen1%
        \box0%
        \hskip -\dimen0}
\newcommand{\LVCenter}[1]{\setbox 0=\hbox{#1}%
        \dimen1=\ht0%
        \divide\dimen1 by 2%
        \lower \dimen1%
        \box0%
        \hskip -\dimen0}

\begin{document}
\sloppy

\title{ A note on the spectral properties of cluster algebras}

\author[Fern\'andez]{Elsa Fern\'andez}
\address{Elsa Fern\'andez, 
Facultad de Ingenier\'{\i}a, Universidad Nacional de la
  Patagonia San Juan Bosco, 9120 Puerto Madryn, Argentina.}
\email{elsafer9@gmail.com}

\author[Platzeck]{Mar\'{\i}a In\'es Platzeck}
\address{Mar\'{\i}a In\'es Platzeck, Instituto de
    Matem\'atica, Universidad Nacional del Sur, 8000, Bah\'{\i}a
    Blanca, Argentina.}
\email{platzeck@uns.edu.ar}

\begin{abstract}
  Let $Q$ be a finite quiver without oriented cycles and $k$ an algebraically closed field.In this paper we establish a connection between cluster algebras and the representation theory of the path algebra $kQ$, in terms of the spectral properties of the quivers mutation equivalent to $Q$.
\end{abstract}

\thanks{The authors thankfully acknowledge partial support from CONICET and
  from Universidad Nacional del Sur, Argentina. M.~I.~Platzeck is a researcher from CONICET, Argentina}

\maketitle

%%%%%%%%%%%%%%%%%%%%%%%%%%%%%%%%%%%%%%%%%%%%%%%%%%
%%%%%%%%%%%%%%%%%%%%%%%%%%%%%%%%%%%%%%%%%%%%%%%%%%
\section{Introduction}
The theory of cluster algebras goes back to the work of S. Fomin and A. Zelevinsky (\cite{FZ1}, \cite{FZ2}). One of the main ingredients in the study of this subject is the notion of mutation of quivers.

On  the other hand, there is the theory of graph spectra.  This theory can be used to obtain information about the graph structure. In particular, spectra of graphs are interesting in representation theory of finite dimensional algebras. Works along these lines are, for example,  \cite{CDS}, \cite{G}, \cite{JAP}, \cite{JAPT}, \cite{T} \cite {F} and \cite{V}.
 
In this paper we associate a numerical invariant to the mutation class of a quiver $Q$, in terms of the spectral radius of $Q$. Our first result gives an alternative characterization of cluster algebras of finite type using this invariant. The second establishes a connection between cluster algebras and the representation theory of path algebras.

%%%%%%%%%%%%%%%%%%%%%%%%%%%%%%%%%%%%%%%%%%%%%%%%%%
%%%%%%%%%%%%%%%%%%%%%%%%%%%%%%%%%%%%%%%%%%%%%%%%%%
\section{Preliminaries}

\subsection{The spectral radius}
\label{subsec:spectral radius}
%%%%%%%%%%%%%%%%%%%%%%%%%%%%%%
Let $\Delta = (\Delta_0, \Delta_1)$ be a finite graph where $\Delta_0$ is the set of vertices and $\Delta_1$ the set of edges. We may assume that $\Delta_0 = \{1, 2, \dots, n\}$. Recall that {\it the adjacency matrix} $A_\Delta$ is the integer $n\times n$ matrix defined as follows:

$$( A_\Delta)_{ij}=\biggl\lbrace
\begin{matrix}
\text{number of edges between $i$ and $j$,}& \text {if }i\neq j\  \\ 
\text{two times the number of loops at $i$,}& \text{ if }  i=j\ .
\end{matrix}
$$ 

Clearly, $A_\Delta$ is a symmetric matrix. Let $\rho (A_\Delta)$ be {\it the spectral radius} of $A_\Delta$, that is, $\rho(A_\Delta) = {\rm max} |\lambda|$, where $\lambda$ runs over all the eigenvalues of $A_\Delta$.

The characteristic polynomial and the spectral radius of $A_\Delta$ are called the characteristic polynomial and the spectral radius of $\Delta$, respectively.
Also, the adjacency matrix and the spectral radius $\rho (Q)$  of a finite quiver $Q$ are defined as the adjacency matrix and the spectral radius of the corresponding underlying graph. 

The following is a useful characterization of simply laced Dynkin diagrams and simply laced extended Dynkin diagrams in terms of the spectral radius.

\begin{prop} \label{Dynkin} {\rm (\cite{CDS})}
Let $\Delta$ be a finite connected graph. Then
\begin{itemize}
 \item[{\rm (a)}]  $\Delta$ is a Dynkin diagram if and only if $\rho (\Delta) < 2$.

  \item[{\rm (b)}]  $\Delta$ is an extended Dynkin diagram if and only if $\rho (\Delta) = 2$.

  \end{itemize}

\end{prop}

The result  is given  as an exercise in \cite{CDS} (see also \cite [V.1.3, Ejercicio 5]{JAP}). A proof is given in \cite{F}.

\vskip.2in
Throughout the paper $k$ denotes an algebraically closed field and all quivers are finite. The next result shows that we can get information about the representation type of $kQ$ in terms of the spectral radius of the associated adjacency matrix.

\begin{thm}  {\rm (\cite{T})}
Let $Q$ be a connected quiver without oriented cycles. Then 

\begin{itemize}
 \item[{\rm (a)}]  $kQ$ is of finite representation type if and only if $\rho (Q) < 2$.

  \item[{\rm (b)}]  $kQ$ is of tame representation type if and only if $\rho (Q) = 2$.

   \item[{\rm (c)}]  $kQ$ is of wild representation type if and only if $\rho (Q) > 2$.
\end{itemize}
\end{thm}
\begin{proof}: The result is a straightforward consequence of \cite{T}. It can be proven also in the following way. Part (a) follows from Proposition \ref{Dynkin} (a)    and from Gabriel's Theorem »\cite{G}, and (b) is a consequence of (b) in the same proposition and the characterization of tame algebras given by Donovan-Freislich (\cite{DF} and Nazarova \cite{N}. Finally, (c) follows from (a) and (b).
\end{proof}

\subsection{Mutations of quivers, cluster algebras and cluster tilted algebras}
\label{subsec:quivers}

The interest in mutation of quivers goes back to the work of Fomin and Zelevinsky in \cite{FZ1} were they introduced cluster algebras.
For a finite quiver without loops or $2$-cycles, and with vertex set $\{ 1,2, \dots , n\}$, we denote by $b_{ij}$ the number of arrows $i \rightarrow j$ minus the number of arrows $j\rightarrow i$ in $Q$. We observe that at least one of these numbers is zero, because $Q$ does not have $2$-cycles.

Let $s$ be a vertex of $Q$. Recall from \cite{FZ1} that, for any vertex $s$, the mutation of $Q$ in direction $s$ is another quiver $\mu_s(Q) = Q'$ defined as follows:

$$ b'_{ij}=\biggl\lbrace
\begin{matrix}
{-b_{ij}  }& \text {if }i=s  \text { or } j=s\  \\ 
b_{ij}+ \text { sgn }(b_{is})[b_{is}b_{sj}]_+,& \text{ otherwise}
\end{matrix}  $$
where $[x]_+ = $ max $\{x,0\}$.

Note that mutation at $s$ is an involution. The mutation class of a quiver $Q$ is the set of all quivers obtained from it by iterated mutations.

We recall that the cluster algebra $\mathcal A(Q)$ associated with the quiver $Q$ is the $\mathbb {Q}$- subalgebra of the field of rational functions $\mathbb Q (x_1, \dots , x_n)$ over generated by a family of distinguished generators called the cluster variables. Here $\mathbb Q$ denotes the field of rational numbers. When the set of cluster variables is finite then $\mathcal A(Q)$ is said to be of finite type. For a precise definition and facts about cluster algebras  we refer the reader to \cite{FZ1}, \cite{FZ2}.

We also recall that cluster algebras associated to quivers which are mutation equivalent are isomorphic.

We will  need the notion of cluster tilted algebra. These algebras were introduced by Buan, Marsh and Reiten  as the endomorphism rings of cluster-tilting  objects in the cluster category of  $kQ$, and their quivers correspond  with the quivers in the mutation class of $Q$   (see \cite {BMRRT}, \cite{BMR2}, \cite{BMR3}).

%%%%%%%%%%%%%%%%%%%%%%%%%%%%%%

%%%%%%%%%%%%%%%%%%%%%%%%%%%%%%%%%%%%%%%%%%%%%%%%%%
\section{The main results}
\label{sec:roll}

In this section we state and prove our main results. Let $Q$ be a connected quiver without loops or $2$-cycles. We denote by $\epsilon (Q)$ the infimum of the spectral radius of the quivers mutation-equivalent to $Q$.

The invariant $\epsilon (Q)$ is not bounded. In fact, $\epsilon (Q_1)=m$ if $Q_1$ is the quiver \xymatrix{ {\bullet} \ar @{->} @< 6pt> [r]
\ar @{.} @< 3pt> [r]
\ar @{.} @< 0pt> [r]
\ar @{.} @< -3pt> [r]
\ar @{->} @< -6pt> [r]& {\bullet}}
with two vertices and $m$ arrows. However, we will show that $\epsilon (Q)<2$ for quivers $Q$ mutation equivalent to a Dynkin quiver.

An important result of Fomin and Zelevinsky states that a cluster algebra $\mathcal A(Q)$ is of finite type if and only if $Q$ is mutation equivalent to an orientation of a simple laced Dynkin diagram. We will give  an alternative  characterization of cluster algebras of finite type in the following theorem.
\vskip .12in
\begin{thm}
Let $Q$ be a connected quiver without loops or $2$-cycles. The following are equivalent
\begin{itemize}
\item[{\rm (a)}]  The cluster algebra $\mathcal A(Q)$ is of finite type. 
\item[{\rm (b)}] $\epsilon(Q) <2$.
\end{itemize}
\end{thm}
\begin{proof}
Assume that $\mathcal A(Q)$ is of finite type. Then there exists a simply laced Dynkin quiver $Q'$ such that $Q$ is mutation equivalent to $Q'$.
By Proposition \ref{Dynkin} we have $\rho(Q') <2$, and therefore $\epsilon (Q)<2$.

Conversely, if (a) does not hold then $\rho(Q') \geq 2$ for every $Q'$ in the mutation class of $Q$. Thus 
$\epsilon(Q) \geq 2$, a contradiction.
\end{proof}
\vskip .07 in

Now we give a result which illustrates the relationship 
 between cluster algebras and the representation theory of path algebras.
\vskip .12in
\begin{thm}
Let $Q$ be a connected quiver without oriented cycles. Then 
\begin{itemize}
\item[{\rm (a)}]  $kQ$ is of finite representation type if and only if $\epsilon (Q) <2$.

  \item[{\rm (b)}]  $kQ$ is of tame representation type if and only if $\epsilon (Q) =2$.
\item[{\rm (c)}]  $kQ$ is of wild representation type if and only if $\epsilon (Q) >2$.
\end{itemize}
\end{thm}
\begin{proof}
(a) Suposse $kQ$ is of finite representation type. Then $\rho(Q) <2$, hence $\epsilon (Q)<2$. To prove the converse we assume that $\epsilon (Q)<2$.
Then there  exists a quiver $Q'$  mutation equivalent to $Q$ such that $\rho (Q') < 2$. By Proposition \ref{Dynkin} we know that $Q'$ is a Dynkin quiver, with type $\underline Q'$. Since $Q$ is in the mutation class of $Q'$, then $Q$ is the quiver of a cluster tilted algebra $C$ of Dynkin type $\underline Q'$. We will show next that $Q$ and $Q'$ are quivers of the same type, so $Q$ is Dynkin and therefore $kQ$ is of finite representation type. We will do this using the notion of relation extension, introduced in \cite{ABS}.

We recall that the relation extension $\mathcal R(B)$ of  a finite dimensional    $k$-algebra $B$  of global dimension at most two was defined in \cite{ABS} as the trivial extension of $B$ by the $B$-$B$-bimodule  $\Ext^2_B (DB,B)$. When $B$ has no oriented cycles, any relation in $B$ gives rise to an oriented cycle in  $\mathcal R(B)$. Thus $\mathcal R(B)$ has oriented cycles, unless $B$ is hereditary, and in this latter case $\mathcal R(B) =B.$

It follows from \cite [3.4]{ABS} that the cluster-tilted algebra $C$ is isomorphic to the relation extension of the tilted algebra $B= \End_{kQ_1}(T)$, where $T$ is a tilting  module over an algebra $kQ_1$ derived equivalent to $kQ$. Then $Q$ and $Q_1$ are quivers of the same Dynkin type, so $B$ is tilted of type $\underline Q$. Since $Q$ has no oriented cycles and  is the quiver of $C \simeq \mathcal R(B)$ it follows that $B$ is hereditary and $\mathcal R(B) = B$. Thus $C\simeq B$ is hereditary, so $C=kQ$.
 On the other hand, $C$ is of Dynkin type $\underline Q'$ and then we conclude that $Q$ and $Q'$ are Dynkin quivers of the same type, as desired.

(b) If $kQ$ is a tame algebra, then by Proposition \ref{Dynkin}, we have $\rho (Q) =2$. Therefore $\epsilon (Q) \leq 2$ and by (a) it  follows that $\epsilon (Q) =2$.

Conversely, let $\epsilon (Q) =2$ and let $ \mathcal M_Q  $ denote the mutation class of $Q$. If there exists a quiver $Q'$  in $\mathcal M_Q$ such that $\rho (Q') =2$, then an argument similar to that used in the proof of (a) leads to the desired statement.
We prove next that such a $Q'$ exists. Assume, to the contrary, that $\rho (Q'') >2$ for any quiver $Q''$  in $\mathcal M_Q$. Since $\epsilon (Q) =2$ then $\rho (Q'') <3$ for infinitely many quivers $Q''$ in $\mathcal M_Q$. 

Let $r$ be a fixed integer and let $\mathcal M_Q^r = \{Q'' \in \mathcal M_Q $ such that $Q''$ has at most $r$ arrows$\}$. The number of vertices of the quivers in $\mathcal M_Q$ is fixed, since it coincides with the number of vertices of $Q$, therefore the set $\mathcal M_Q^r$ is finite. Choosing $m$ large enough we have that any  quiver in $\mathcal M_Q \setminus  \mathcal M_Q^m$ contains   \xymatrix{ {\bullet} \ar @{->} @< 4pt> [r]
\ar @{->} @< 0pt> [r]
\ar @{->} @< -4pt> [r]& {\bullet}}
 as a subquiver. Since by Corollary 1.3 in \cite [V]{JAP}   the spectral radius of a subquiver of a given quiver can not exceed the radius of the quiver,  and $\rho (\xymatrix{ {\bullet} \ar @{->} @< 4pt> [r]
\ar @{->} @< 0pt> [r]
\ar @{->} @< -4pt> [r]& {\bullet}}
) =3$, then $\rho (Q'')\geq 3$ for any $Q''$ in  $\mathcal M_Q \setminus  \mathcal M_Q^m$ .

Therefore $2 = \epsilon (Q) = {\rm inf}  \, \{ \rho(Q''): Q'' \in  \mathcal M_Q^m \}$. Since  $\mathcal M_Q^m$ is a finite set there exists $Q' \in  \mathcal M_Q^m$ such that $\rho (Q') = 2$. This contradicts the assumption that $\rho (Q'') >2$ for any quiver $Q''$  in $\mathcal M_Q$, and ends the proof of (b).

(c) follows from (a) and (b).
\end{proof}
\vskip .2in
\begin{cor} Let $Q$ be a wild quiver without oriented cycles. Then any quiver in the mutation class of $Q$ is wild.

\end{cor}

%%%%%%%%%%%%%%%%%%%%%%%%%%%%%%

%%%%%%%%%%%%%%%%%%%%%%%%%%%%%%%%%%%%%%%%%%%%%%%%%%%%%%%%%%%%

\end{document}